\documentclass[draft]{article}
\usepackage{amsmath,amsfonts,latexsym, amssymb,amsrefs}
\usepackage{dsfont}
\usepackage{stmaryrd}
\usepackage{enumerate}

\usepackage{color}

\newcommand{\ii}{{\mathrm{i}}}
\DeclareMathOperator{\OO}{O}
\DeclareMathOperator{\oo}{o}
\newcommand{\var}{{\mathrm{var}}}
\newcommand{\wt}{\widetilde}

\newcommand{\e}{\varepsilon}

\newcommand{\rd}{{\rm d}}
\newcommand{\bR}{{\mathbb R}}

\newcommand{\tgamma}{{\widetilde \gamma}}

\renewcommand{\L}{\mbox{L}}

\renewcommand\Im{\operatorname{Im}}

\newcommand{\bv}{{\bf{v}}}
\newcommand{\bw}{{\bf{w}}}

\newcommand{\bG}{{\bf {G}}}

\newcommand{\bla}{\mbox{\boldmath $\lambda$}}

\newcommand{\al}{\alpha}

\newcommand{\be}{\begin{equation}}
\newcommand{\ee}{\end{equation}}

\newcommand{\cQ}{{\mathcal Q}}

\newcommand{\cR}{{\mathcal R}}

\newcommand{\E}{{\mathbb E }}
\newcommand{\R}{{\mathbb R }}

\renewcommand{\P}{{\mathbb P}}

\newcommand{\ind}{{\,\mathrm{d}}}

\newtheorem{theorem}{Theorem}

\newtheorem{lemma}[theorem]{Lemma}
\newtheorem{proposition}[theorem]{Proposition}
\newtheorem{definition}[theorem]{Definition}

\newcommand{\qed}{\hfill\fbox{}\par\vspace{0.3mm}}
\newenvironment{proof}{{\bf Proof.}} {\hfill\qed}

\numberwithin{equation}{section}
\numberwithin{theorem}{section}
\numberwithin{remark}{section}

\usepackage{geometry}     
\usepackage{graphicx}

\numberwithin{equation}{section}
\def\cal{}
\def\RR{{\mathbb R}}

\def\ZZ{{\mathbb Z}}

\def\NN{{\mathbb N}}
\def\CC{{\mathbb C}}

\def\cH{{\mathcal H}}

\def\W2{W^{1,2}({\cal O}(M))}

\def\1half{\frac{1}{2}}

\title{Bulk Universality of General $\beta$-Ensembles with Non-convex Potential}

\date{Aug 21, 2012}

\author{Paul Bourgade${}^1$\quad
L\'aszl\'o Erd\H os${}^2$\thanks{Partially supported
by SFB-TR 12 Grant of the German Research Council} \quad
Horng-Tzer Yau${}^1$\thanks{Partially supported
by NSF grants DMS-0757425, 0804279}
 \\\\
Department of Mathematics, Harvard University\\
Cambridge MA 02138, USA \\   bourgade@math.harvard.edu \quad
 htyau@math.harvard.edu ${}^1$ \quad  \\ \\
Institute of Mathematics, University of Munich, \\
Theresienstrasse 39, D-80333 Munich, Germany \\ lerdos@math.lmu.de ${}^2$
\\}

\begin{document}

\maketitle

\vspace{0.7cm}
\begin{center}
{\it \large Dedicated to Elliott H. Lieb on the occasion of his 80th birthday}
\end{center}
\vspace{0.7cm}

\begin{abstract}
We prove the bulk  universality of the $\beta$-ensembles with non-convex
regular analytic potentials for any $\beta>0$. This removes the convexity
assumption appeared in the earlier work \cite{BouErdYau2011}.  The
convexity condition  enabled us  to use the  logarithmic Sobolev inequality to estimate
events with small probability.  The  new idea
is to introduce  a ``convexified measure''
so that the local statistics are preserved under this convexification.
\end{abstract}

\vspace{1.5cm}

{\bf AMS Subject Classification (2010):} 15B52, 82B44

\medskip

\medskip

{\it Keywords:} $\beta$-ensembles, universality,  log-gas.

\medskip

\centerline{\date}

\newpage

\section{Introduction and  the main results}

The  classical invariant ensembles of  random matrices are given by   probability  measures
of the form $e^{- N \beta  { \rm Tr} V(H)/2 }$
where $N$ is the size of the matrix $H$ and $V$ is a real valued potential.
The parameter  $\beta=1,2,4$ is determined by the symmetry type of the matrix,
corresponding respectively to the classical orthogonal, unitary  or symplectic ensemble.
Let $\lambda = (\lambda_1, \lambda_2, \ldots ,\lambda_N)\in \Sigma_N$ be the
eigenvalues of $H$ in increasing order, where $\Sigma_N\subset \RR^N$ denotes the simplex
determined by $\lambda_1\le \lambda_2\le \ldots \le\lambda_N$.
It is well-known that the probability distribution of the ordered eigenvalues on
 $\Sigma_N$ is given by
\begin{equation}\label{eqn:measure}
\mu^{(N)}_{\beta, V}=\mu^{(N)}(\rd\lambda)=\frac{1}{Z_N}
e^{- \beta N \cH(\lambda)}\rd\lambda,
\qquad
\cH(\lambda) =   \sum_{k=1}^N  \frac{1}{2}V(\lambda_k)-
\frac{1}{N} \sum_{1\leq i<j\leq N}\log (\lambda_j-\lambda_i) .
\end{equation}
For non-classical values of $\beta > 0$, i.e., $\beta \not \in \{1, 2, 4\}$,
one can still consider the measure (\ref{eqn:measure}) on $\Sigma_N$,
 but in general there is no simple  natural matrix model producing this measure
except for the Gaussian case, $V(x)=x^2$, which corresponds to  a tri-diagonal  random matrix
\cite{DE,VV2009}.
  We will view $\mu=\mu^{(N)}$ as a Gibbs measure of particles
in $\bR$  with a logarithmic interaction, where the parameter $\beta > 0 $
is interpreted as the inverse temperature.  We will
refer to the variables $\lambda_j$ as particles or points
and the system is called log-gas  or general $\beta$-ensemble.

The universality conjecture asserts that the eigenvalue gap distributions
 in the bulk  depend only on $\beta$
and are independent of the potential $V$. For classical ensembles, the eigenvalue correlation functions
can be explicitly expressed
in terms of polynomials orthogonal to the measure $e^{- \beta V(x)/2}$.
Thus the analysis of the correlation functions
relies heavily on the asymptotic properties of the  corresponding orthogonal polynomials.
This approach, initiated by Dyson, Gaudin and Mehta (see \cite{Meh1991} for a review)
  was the starting point for
all results on classical universality.
Precise analysis on  orthogonal polynomials for general class of weight functions
was made possible by the Riemann-Hilbert approach
\cite{BleIts1999, DeiKriMcLVenZho1999I, DeiKriMcLVenZho1999II}.
  There are also methods independent of
the Riemann-Hilbert  approach, see, e.g.,
\cite{PasShc1997, PasShc2008,  Lub2009}.
The  universality for $\beta=2$ was proved for very general potential.
For $\beta = 1, 4$ \cite{DeiGio2009, KriShc2011, Shc2011}
it was proved for  analytic $V$ with some additional conditions.
 A summary of recent developments can be found in
\cite{AndGuiZei2010, Dei1999, DeiGio2009,PasShc2011}.

For non-classical values of $\beta$, i.e., $\beta \not \in \{1, 2, 4\}$,
there is no simple expression of the correlation functions in
terms of orthogonal polynomials.  In \cite{BouErdYau2011}, we initiated a new approach to prove
bulk universality for all $\beta > 0$ and strictly convex $V$.
The method was  based on estimating correlation functions by local Dirichlet form
and the main ingredients consist of the following  two steps:

\medskip
\noindent
{\it Step 1.  Rigidity of eigenvalues.}
This establishes that the location of the eigenvalues  are not too far
from their classical locations  determined by the equilibrium density $\rho(s)$.

\medskip
\noindent
{\it Step 2.  Uniqueness of local Gibbs measures with logarithmic interactions.}
With the precision of eigenvalue location estimates from  Step 1 as an input,
the eigenvalue gap distributions are
 shown to be  given by the corresponding
Gaussian ones. (We will take the uniqueness of the gap
distributions as our definition of the uniqueness of Gibbs state.)

Our goal is to extend this result to the non-convex case.
It was emphasized in \cite{BouErdYau2011} that the convexity of the potential $V$ was used only in Step 1.
So in order to apply this method, it suffices to prove the rigidity estimate which
we now introduce.

We will assume that the potential $V$ is  real analytic function   in $\RR$  such that
its second derivative is bounded below, i.e. we have
\begin{equation}\label{eqn:LSImu}
\inf_{x\in\RR}V''(x) \ge -2W
\end{equation}
for some constant $W\ge 0$, and
\begin{equation}\label{eqn:GrowthCondition}
V(x)> (2 + \alpha)\ln(1+|x|),
\end{equation}
for some\ $\alpha > 0$,
if $|x|$ is large enough.
It is known \cite{BouPasShc1995} that under these (in fact, even weaker) conditions
the measure is normalizable, $Z_N<\infty$. Moreover,
the averaged density of the empirical spectral measure, defined as
$$
   \rho^{(N)}_1(\lambda)= \rho_1^{(N,\beta,V)}(\lambda) : = \E_{\mu^{(N)}} \frac{1}{N}\sum_{j=1}^N \delta(\lambda-\lambda_j)
$$
converges weakly to a continuous function $\rho$, the equilibrium density, with compact support.
We additionally assume that   $\rho(s)$  is supported on a single interval $[A,B]$, and that $V$
 is {\it regular} in the sense of \cite{KuiMcL2000}. We recall that $V$ is regular if
its equilibrium  density $\rho$ is positive on $(A,B)$
 and vanishes like a square
root\footnote{This is not a strong constraint:
\cite{KuiMcL2000} proves that the regular potentials $V$ are a dense and open subset of the potentials
for the topology induced by the distance
$$
d(V,W)=\sum_{j=0}^3\sum_{k=1}^\infty 2^{-k}\frac{\|V^{(j)}-W^{(j)}\|_{\L^{\infty}[-k,k]}}
{1+\|V^{(j)}-W^{(j)}\|_{\L^{\infty}[-k,k]}}
+
\sum_{k=1}^\infty2^{-k}\frac{|G_k(V)-G_k(W)|}{1+|G_k(V)-G_k(W)|},
$$
where $G_k(V)=\inf_{|x|>k}V(x)/\log|x|$.} at each of the endpoints of $[A,B]$, that is
\begin{align}\label{sqsing}
\rho(t)&=s_A\sqrt{t-A}\left(1+\OO\left(t-A\right)\right),\ t\to A^+,\\
\rho(t)&=s_B\sqrt{B-t}\left(1+\OO\left(B-t\right)\right),\ t\to B^-, \nonumber
\end{align}
for some constants $s_A,\, s_B>0$.

In this paper, we are interested in the usual $n$-point correlation functions,
generalizing $\rho_{1}^{(N)}$,
and  defined by
\begin{equation}\label{eqn:corrFunct}
\rho^{(N)}_n(x_1,\ldots,x_n)=
\int_{\RR^{N-n}}\tilde\mu(x)\rd x_{n+1}\dots\rd x_{N},
\end{equation}
where $\tilde \mu$ is the  symmetrized version of $\mu$ given in \eqref{eqn:measure}
but defined  on $\RR^N$  instead of the simplex $\Sigma_N$:
$$
\tilde\mu^{(N)}(\rd\lambda)=\frac{1}{N!}\mu(\rd\lambda^{(\sigma)}),
$$
where
$\lambda^{(\sigma)}=(\lambda_{\sigma(1)},\dots,\lambda_{\sigma(N)})$, with
$\lambda_{\sigma(1)}<\dots<\lambda_{\sigma(N)}$.

In the following, we omit the superscript $N$ and we will write $\mu$ for $\mu^{(N)}$.
We will use $\P_\mu$ and $\E_\mu$ to denote the probability and the
expectation with respect to $\mu$.
Let the {\it classical position}  $\gamma_k$ be
defined by
\begin{equation}\label{gammadef}
 \int_{-\infty}^{\gamma_k}\rho(s)\rd s=\frac{k}{N}.
\end{equation}
Finally, we introduce the notation $\llbracket p,q\rrbracket = [p,q]\cap \ZZ$
for any  real numbers $p<q$.

It is known that the particles are rigid, i.e. they
 cannot be far from their classical
locations\footnote{
For eigenvalues in the bulk, (\ref{eqn:largDev1}) follows from the large
 deviations for the empirical spectral measure with speed $N^2$
\cite{BenGui1997,AndGuiZei2010}, and for the extreme eigenvalues the large deviations principle
with speed $N$
was proved in \cite{AndGuiZei2010}, Theorem 2.6.6, up to a condition on the partition function
that follows from Theorem 1 (iii) in \cite{Shc2011}.}: for any
 $\e>0$ there are positive constants $c_1$, $c_2$ such that, for all $N\geq 1$,
\begin{align}
\P_\mu\left( \exists k\in\llbracket1,N\rrbracket\mid
 | \lambda_k-  \gamma_k| \ge \e \right)\leq c_1 e^{-c_2 N}. \label{eqn:largDev1}
\end{align}
The main technical result of this paper
 is to prove that rigidity holds
for the measure $\mu$ at the optimal scale $1/N$ in the bulk in the following sense.
This theorem extends  our rigidity result in \cite{BouErdYau2011} to non-convex potential  $V$.

\begin{theorem}[Rigidity estimate in the bulk]\label{thm:rigidity} Let $V$ be  real analytic,
regular with equilibrium density
supported on a single interval $[A,B]$,  and satisfy  (\ref{eqn:LSImu}), (\ref{eqn:GrowthCondition}).
Take any $\alpha>0$ and $\e>0$. Then there are constants
$\delta,c_1,c_2>0$ such that for any $N\geq 1$ and $k\in\llbracket \alpha N,(1-\alpha) N\rrbracket$,
\begin{equation}\label{bulkrig}
\P_\mu\left(|\lambda_k-\gamma_k|> N^{-1+\e}\right)\leq c_1e^{-c_2N^\delta}.
\end{equation}
\end{theorem}

 Our main result on the universality is the following theorem:

\begin{theorem}[Bulk universality]\label{thm:Main} Let $V$ be  real analytic,
regular with equilibrium density
supported on a single interval $[A,B]$,  and satisfy  (\ref{eqn:LSImu}), (\ref{eqn:GrowthCondition}).
 Then for any $\beta>0$ the bulk universality holds for the
  $\beta$-ensemble $\mu=\mu_{\beta, V}$.
 More precisely, for any $E\in (A,B)$ and $|E'|<2$,
for any smooth test functions $O$ with compact support  and for any $0<k\le \frac{1}{2}$,
we have,  with $s:=N^{-1+k}$, that
\begin{align*}
\lim_{N \to \infty} \int  & \rd \alpha_1 \cdots \rd \alpha_n\, O(\alpha_1,
\dots, \alpha_n) \Bigg [
  \int_{E - s}^{E + s} \frac{\rd x}{2 s}  \frac{1}{ \rho (E)^n  }  \rho_n^{(N)}   \Big  ( x +
\frac{\alpha_1}{N\rho(E)}, \dots,   x + \frac{\alpha_n}{N\rho(E)}  \Big  ) \\
&
-   \int_{E' - s}^{E' + s} \frac{\rd x}{2 s}  \frac{1}{\rho_{sc}(E')^n}
\rho_{{\rm Gauss}, n}^{(N)}
   \Big  ( x +
\frac{\alpha_1}{N\rho_{sc}(E')}, \dots,   x + \frac{\alpha_n}{N\rho_{sc}(E')}  \Big  ) \Bigg ]
\;=\; 0\, .
\end{align*}
Here $\rho_{sc}(E)=\frac{1}{2\pi}\sqrt{4-E^2}$ is the Wigner semicircle law
and $\rho_{{\rm Gauss}, n}^{(N)}$ are the correlation functions of the Gaussian $\beta$-ensemble, i.e.
with $V(x)=x^2$.
\end{theorem}

Theorem \ref{thm:Main} follows immediately from the rigidity estimates, \eqref{eqn:largDev1},
\eqref{bulkrig},
and the uniqueness of local Gibbs measure, i.e.,  Theorem 2.1 and Corollary 2.2  in \cite{BouErdYau2011}.
We note that the proof of the latter results in Section 4 of \cite{BouErdYau2011}
uses only the rigidity estimate, given in Theorem 3.1 of \cite{BouErdYau2011}, as an input.
Once the rigidity estimate is proven, the rest of the argument is identical
and  we will not repeat it here.

The rest of this paper is devoted to the proof of Theorem \ref{thm:rigidity}.
After some initial estimates concerning the large deviations regime and global smooth linear statistics
(Section 2), the proof consists in the following steps. First we compare
$\mu$ to some {\it convexified measures} $\nu$ (Section 3); the Hamiltonian $\cH_\nu$ of $\nu$
differs from that of $\mu$ mainly by some properly chosen linear statistics of the $\lambda_i$'s,
allowing $\cH_\nu$ to be convex.  Despite this change in convexity, we will prove that
the two measures $\mu$
and $\nu$ have the same  subexponentially small probability events.
 This step is the main extra ingredient allowing one to generalize the rigidity estimate
obtained in \cite{BouErdYau2011}.
Then by a self-improving method,
this measure $\nu$ (together with $\mu$) is proved to have rigidity till the optimal scale,
thanks to comparisons with locally constrained versions of $\nu$ (Section 4).

\section{Preliminary results}

\subsection{Equilibrium measure, large deviations}

For  analytic potential  $V$ satisfying the asymptotic growth condition (\ref{eqn:GrowthCondition}),   the equilibrium measure $\rho(s)\rd s$ associated with  $(\mu^{(N)})_{N\geq 0}$
can be defined as the unique minimizer (in the set of probability measures on $\RR$ endowed with the weak topology) of the
functional
$$
I(\nu)=
\int V(t)\rd\nu(t)-
\iint\log|t-s|\rd\nu(s)\rd\nu(t)
$$
if $\int V(t)\rd\nu(t)<\infty$, and $I(\nu)=\infty$ otherwise.
Moreover, if one assumes that $\rho$ is supported on a single interval $[A,B]$ and regular in the sense of the previous section, $\rho$
has the following properties:
\begin{enumerate}[(a)]
\item This equilibrium measure satisfies
\be
   \frac{1}{2}V'(t) = \int \frac{\rho(s)\rd s}{t-s}.
\label{equilibrium}
\ee
for any $t\in(A,B)$.
\item For any $t\in[A,B]$,
\begin{equation}\label{eqn:rho}
\rho(t)\ind t=\frac{1}{\pi}r(t)\sqrt{(t-A)(B-t)}\mathds{1}_{[A,B]}\ind t,
\end{equation}
where $r$ can be extended into an analytic function in $\CC$ satisfying
\begin{equation}\label{eqn:r}
r(z)=\frac{1}{2\pi}\int_A^B\frac{V'(z)-V'(t)}{z-t}\frac{\rd t}{\sqrt{(t-A)(B-t)}}.
\end{equation}
\end{enumerate}

In order to have the density supported strictly in a compact interval,
for given $\kappa>0$,
define the following variant of $\mu^{(N)}$ conditioned to have all particles in $[A-\kappa,B+\kappa]$:
\begin{equation}\label{eqn:truncMeasure}
\mu^{(N,\kappa)}(\rd\lambda)=\frac{1}{Z_{N,\kappa}}
\prod_{1\leq i<j\leq N}|\lambda_i-\lambda_j|^\beta\prod_{k=1}^N e^{-N\frac{\beta}{2}V(\lambda_k)}\mathds{1}_{\lambda_k\in[A-\kappa,B+\kappa]}
\rd\lambda_1\dots\rd\lambda_N.
\end{equation}

 In this paper we will choose $\kappa$ to be small.
This choice differs from \cite{BouErdYau2011} where, instead of $[A-\kappa,B+\kappa]$, we restricted
the particles to $[-R,R]$ for a very large $R$. The smaller interval is needed here
because we need $r$ to be positive on the support of
$\mu^{(N,\kappa)}$ in the proof of Lemma \ref{lem:Johansson}.   Unlike in the case of convex $V$ where
 $r$ is known to have no real zero at all, for the non-convex regular case  we only know that $r$ is nonzero in the interval
$[A,B]$. By continuity, it is also nonzero in $[A-\kappa,B+\kappa]$ for some small $\kappa$.

Let  $\rho_k^{(N,\kappa)}$ denote the correlation functions of the measure $\mu^{(N,\kappa)}$.
Then Lemma 1 in \cite{BouPasShc1995} states that
under condition (\ref{eqn:GrowthCondition}), for some large enough $\kappa$ there exists
 some $c>0$, depending only on $V$,  such that for any $x_1,\dots,x_k\in[A-\kappa,B+\kappa]$, we have
\begin{equation}\label{eqn:BPS1}
\left|\rho^{(N,\kappa)}_k(x_1,\dots,x_k)-\rho^{(N)}_k(x_1,\dots,x_k)\right|
\leq \rho^{(N,\kappa)}_k(x_1,\dots,x_k)e^{-c N},
\end{equation}
and for $x_1,\dots,x_j\not\in[A-\kappa,B+\kappa]$, $x_{j+1},\dots,x_k\in[A-\kappa,B+\kappa]$,
\begin{equation}\label{eqn:BPS2}
\rho^{(N)}_k(x_1,\dots,x_k)\leq e^{-c N\sum_{i=1}^{j}\log |x_i|}.
\end{equation}
The estimates (\ref{eqn:BPS1}) and (\ref{eqn:BPS2})
actually also hold for arbitrarily small fixed $\kappa>0$
thanks to the large deviations estimates (\ref{eqn:largDev1}).

\subsection{Linear statistics}\label{subsec:LinStat}

The following lemma was essentially proven in \cite{Shc2011} (for the variance of linear statistics).
\begin{lemma}\label{lem:LinStat}
For any function $\phi$ with $
\|\phi\|_\infty+\|\phi'\|_\infty+\|\phi''\|_\infty<\infty$, there is a constant $c>0$ depending only on $V$ and $\phi$ (one can choose $c=\OO(
\|\phi\|_\infty+\|\phi'\|_\infty+\|\phi''\|_\infty)$) such that, for any
$N\geq 1$ and $s>0$,
$$
\P_\mu\left(\left|\sum_{i=1}^N\phi(\lambda_i)-N\int_{\RR}\rho(u)\phi(u)\rd u\right|>s\right)\leq e^{-c s/\log N}.
$$
\end{lemma}

\begin{proof}
Without loss of generality, we can assume that $\phi$ is
compactly supported (thanks to large deviation estimates such as (\ref{eqn:BPS2})).
We know from Shcherbina, equation (2.22) in \cite{Shc2011}, that for the Stieltjes transforms, i.e.
$g(u)=1/(z-u)$, there is a constant $c>0$ depending only on $V$ and $g$ (one can choose $c=\OO(\|g^{(4)}\|_\infty)$) such that, for any
$N\geq 1$,
\begin{equation}\label{eqn:diffExp}
\left|\E_{\mu_h}\left(\sum_{i=1}^N g(\lambda_i)-N\int_{\RR}\rho(u)g(u)\rd u\right)\right|\leq c\log N,
\end{equation}
where $\mu_h$ is obtained by replacing $V$ by $V+\frac{h}{N}$ in the definition of
 $\mu$, and $h$ is for example
any $N$-independent smooth compactly supported function.
We will now prove that this implies that (\ref{eqn:diffExp})
actually holds when replacing $g$ by any smooth compactly supported $\phi$, for example by a
Helffer-Sj\"ostrand type
argument, similar to Lemma \ref{lem:HS}.
We can now apply formula (B.13) in \cite{ErdRamSchYau2010} for the
signed measure $\wt\rho=\rho^{(N,\mu_h)}_1-\rho$, with
Stieltjes transform $S$, where $\rho^{(N,\mu_h)}_1$ is the one-point correlation
function of $\mu_h$. We obtain
\begin{align}\label{eqn:HSbound1}
\left|\int_{-\infty}^\infty \phi(\lambda)\wt\rho(\lambda)\rd\lambda\right|
\leq&
C\left|\iint y \phi''(x)\chi(y)\Im S(x+\ii y)\rd x\rd y\right|\\
&+\label{eqn:HSbound2}
C\iint\left(|\phi(x)|+|y||\phi'(x)|\right)|\chi'(y)|\left|S(x+\ii y)\right|\rd x\rd y,
\end{align}
for some universal $C>0$, and
where $\chi$ is a smooth cutoff function with support in $[-1, 1]$, with $\chi(y)=1$ for $|y| \leq 1/2$ and with bounded derivatives.
Note that $\chi'$ is supported on $1/2<|y|<1$ and $\phi,\phi'$
on compact sets, and that $S$ is
uniformly $\OO\left(\frac{\log N}{N}\right)$ on this compact
integration domain, by (\ref{eqn:diffExp}), so
the term (\ref{eqn:HSbound2}) is easily bounded by
$\OO(\|\phi\|_\infty+\|\phi'\|_\infty)\frac{\log N}{N}$.
Concerning the term $(\ref{eqn:HSbound1})$, an easy calculation yields the bound
$\frac{\rd}{\rd y} (y\,\Im S)=\OO(1/y)$,
so integrating from 1 to $y$ we get $|y\,\Im S(x+\ii y)|=\OO(|\log y|)\frac{\log N}{N}$, which is integrable, so
$(\ref{eqn:HSbound1})$ is $\OO(\log N/N)$ as well, finally
proving that (\ref{eqn:diffExp}) holds when replacing $g$
by $\phi$.

Following now Lemma 1 in \cite{Shc2011}, consider
$$
Z_N(t)=\E_{\mu}\left(\exp\left(\frac{t}{\log N}\left(\sum_{i=1}^N\phi(\lambda_i)-N\int_{\RR}\rho(s)\phi(s)\rd s\right)\right)\right).
$$
Then obviously
$
\frac{\rd^2}{\rd t^2}\log Z_N(t)\geq 0
$, so
\begin{multline*}\log Z_N(t)=\log Z_N(t)-\log Z_N(0)\leq |t|\frac{\rd}{\rd t}\log Z_N(t)\\
=
\frac{|t|}{\log N}\E_{\mu_{t\phi/\log N}}\left(\sum_{i=1}^N\phi(\lambda_i)-N\int_{\RR}\rho(s)\phi(s)\rd s
\right),
\end{multline*}
so using (\ref{eqn:diffExp}) we get that $Z_N(t)\leq e^{c|t|}$, from which
Lemma \ref{lem:LinStat} easily follows.
\end{proof}

\subsection{Analysis of the loop equation}

This section analyzes the loop equation \eqref{eqn:firstLoop}
in the following Lemma \ref{lem:Johansson}.
Its proof is  very similar to \cite{BouErdYau2011} except that,  instead of
the logarithmic Sobolev inequality which was valid only for convex $V$, we will  use Lemma \ref{lem:LinStat}.
Furthermore, since the support
of the restricted measure $\mu^{(N,\kappa)}$ has changed,  the integration contours in \eqref{eqn:contour}
are  chosen slightly  differently from those in  \cite{BouErdYau2011}.

In the form presented here, we follow closely the proof in \cite{Shc2011}.
We now introduce some notations needed in the proof.
\begin{itemize}
\item $m_N$ is the Stieltjes transform of $\rho^{(N)}_1(s)\rd s$, evaluated at some $z$ with $\Im(z)>0$, and $m$ its
limit:
$$
m_N(z)=\E_\mu\left(\frac{1}{N}\sum_{k=1}^N\frac{1}{z-\lambda_i}\right)=\int_\RR\frac{1}{z-t}\rho_1^{(N)}(t)\rd t,\
m(z)=\int_\RR\frac{1}{z-t}\rho(t)\rd t.
$$

\item $s(z)=-2r(z)\sqrt{(A-z)(B-z)}$, where the square root is defined such that
$$
f(z)=\sqrt{(A-z)(B-z)}\sim z \quad \text{ as }  \quad z\to\infty;
$$
\item $b_N(z)$ is defined by
$$
b_N(z)=\int_{\RR}\frac{V'(z)-V'(t)}{z-t}(\rho_1^{(N)}-\rho)(t)\ind t;
$$
\item finally, $c_N(z)=\frac{1}{N^2}k_N(z)+\frac{1}{N}\left(\frac{2}{\beta}-1\right)m_N'(z)$, where
$$
k_N(z)=\var_\mu\left(\sum_{k=1}^N\frac{1}{z-\lambda_i}\right).
$$
Here the $\var$ of a complex random variable
denotes $\var(X)=\E(X^2)-\E(X)^2$, i.e. without absolute value
unlike the usual variance.
 Note
that $|\var(X)|\leq \E(|X-\E(X)|^2)$.
\end{itemize}
 The loop equation (see \cite{Joh1998,Eyn2003,Shc2011} for various proofs) is
\begin{equation}\label{eqn:firstLoop}
(m_N-m)^2+s(m_N-m)+b_N=c_N.
\end{equation}
In the  regime where
 $|m_N - m|$  is small, we can neglect the quadratic term. The term $b_N$ is the same order as
$|m_N-m|$ and is  difficult  to treat. As observed in \cite{AlbPasShc2001,Shc2011}, for analytic  $V$ (hence analytic $b_N$), this term vanishes
when we perform a contour integration. So we have roughly the relation
\be\label{55}
(m_N-m) \sim \frac 1 { N^2} \var_\mu\left(\sum_{k=1}^N\frac{1}{z-\lambda_k}\right),
\ee
where we dropped the less important error involving $m_N'(z)/N $ due to the extra $1/N$ factor.
With no convexity assumption on $V$, the difficulty will be to estimate the
above variance to immediately obtain an estimate on $m_N - m$; this is the reason why we will introduce
a convexified version of the measure $\mu$ in the next Section \ref{sec:convexification}.
To quantify more precisely (\ref{55}) we will use the following result, already proved in \cite{BouErdYau2011}
for convex  $V$.

\begin{lemma}\label{lem:Johansson}
Let $\delta>0$.
For $z=E+\ii \eta$ with $A+\delta<E<B-\delta$
assume that
\begin{equation}\label{eqn:kNTo0}
\frac{1}{N^2}k_N(z)\to 0
\end{equation}
 as $N\to\infty$ uniformly in $\eta\geq N^{-1+a}$ for some $0<a<1$. Then there are constants $c,\kappa>0$  such that for any
$N^{-1+a}\leq\eta\leq\kappa$,
$A+\delta<E<B-\delta$,
\begin{equation}\label{eqn:lemJohansson}
|m_N(z)-m(z)|\leq c\left(\frac{1}{N\eta}+\frac{1}{N^2}k_N(z)\right).
\end{equation}
\end{lemma}

\begin{proof}
First, for technical contour integration reasons, it will be easier to consider the measure (\ref{eqn:truncMeasure})
instead of $\mu^{(N)}$ here. More precisely, define
\begin{align*}
m^{(\kappa)}_N(z)&=
\E_{\mu^{(N,\kappa)}}\left(\frac{1}{N}\sum_{k=1}^N\frac{1}{z-\lambda_i}\right)=
\int_\RR\frac{1}{z-t}\rho_1^{(N,\kappa)}(t)\rd t,\\
k_{N}^{(\kappa)}(z)&=\var_{\mu^{(N,\kappa)}}\left(\sum_{k=1}^N\frac{1}{z-\lambda_i}\right),\\
c_N^{(\kappa)}(z)&=\frac{1}{N^2}k_N^{(\kappa)}(z)+\frac{1}{N}\left(\frac{2}{\beta}-1\right){m_N^{(\kappa)}}'(z).
\end{align*}
Then it is a direct consequence of (\ref{eqn:BPS1}) and (\ref{eqn:BPS2}) that for any $\kappa>0$ there is a constant $c>0$ such that
uniformly on
$\eta\geq N^{-10}$ (or any power of $N$),
\begin{equation}\label{eqn:expDiff}
|m^{(\kappa)}_N-m_N|=\OO\left(e^{-c N}\right),\ \ \ |k_{N}^{(\kappa)}-k_{N}|=\OO(e^{-c N}).
\end{equation}
From now, we choose a fixed $\kappa>0$ such that all the zeros of $r$ are at distance at least $10\kappa$ from $[A,B]$ (this is possible because $V$
is regular).
Consider the rectangle with vertices
$
B+5\kappa+\ii N^{-10},A-5\kappa+\ii N^{-10}, A-5\kappa-\ii N^{-10}, B+5\kappa-\ii N^{-10}
$,
call $\mathcal{L}$ the corresponding clockwise closed contour and  $\mathcal{L}'$
the one consisting only in the horizontal pieces, with the same orientation.
{F}rom $(\ref{eqn:firstLoop})$, we obviously have, for $z\not\in\mathcal{L}'$,
$$
\frac{1}{ 2\pi\ii}\int_{\mathcal{L}'}\frac{(m_N(\xi)-m(\xi))^2+s(\xi)(m_N(\xi)-m(\xi))
+b_N(\xi)-c_N(\xi)}
{r(\xi)(z-\xi)}\rd\xi=0.
$$
Note that the above expression makes sense for large enough $N$, because then $r$ has no zero on $\mathcal{L}$.
Using (\ref{eqn:expDiff}), this implies, for $\eta\geq N^{-1}$,
$$
\frac{1}{2\pi\ii}\int_{\mathcal{L}'}\frac{(m^{(\kappa)}_N(\xi)-m(\xi))^2
+s(\xi)(m^{(\kappa)}_N(\xi)-m(\xi))
+b_N(\xi)-c^{(\kappa)}_N(\xi)}{r(\xi)(z-\xi)}\rd\xi=
\OO(e^{-c N}).
$$
Now, as $\rho_1^{(N,\kappa)}$ and $\rho$ are supported on $[A-\kappa,B+\kappa]$, $m_N^{(\kappa)}-m$
and $c_N^{(\kappa)}$ are uniformly $\OO(1)$ in the vertical segments of $\mathcal{L}$. Consequently, from the above equation
$$
\frac{1}{2\pi\ii}\int_{\mathcal{L}}\frac{(m^{(\kappa)}_N(\xi)-m(\xi))^2+s(\xi)(m^{(\kappa)}_N(\xi)-m(\xi))+b_N(\xi)-c^{(\kappa)}_N(\xi)}{r(\xi)(z-\xi)}\rd\xi=
\OO(N^{-10}).
$$
As $b_N$ and $r$ are analytic inside $\mathcal{L}$, for $z$ outside $\mathcal{L}$ we get
$$
\frac{1}{2\pi\ii}\int_{\mathcal{L}}\frac{(m^{(\kappa)}_N(\xi)-m(\xi))^2+s(\xi)(m^{(\kappa)}_N(\xi)-m(\xi))-c^{(\kappa)}_N(\xi)}{r(\xi)(z-\xi)}\rd\xi=
\OO(N^{-10}).
$$
Remember we define
$f(z)=\sqrt{(A-z)(B-z)}$ uniquely by $f(z)\sim z$ as $z\to\infty$. Moreover, $|m_N^{(\kappa)}-m|(z)=\OO(z^{-2})$ as $|z|\to\infty$ because
$\rho$ and $\rho_1^{(N,\kappa)}$ are compactly supported:
\begin{multline*}
|m_N^{(\kappa)}(z)-m(z)|=\left|\int_{A-\kappa}^{B+\kappa}\frac{\rho(t)-\rho^{(N,\kappa)}(t)}{z-t}\rd t\right|\\=\left|\int_{A-\kappa}^{B+\kappa}(\rho(t)-\rho^{(N,\kappa)}(t))\left(\frac{1}{z}+\OO\left(\frac{1}{z^2}\right)\right)\rd t\right|=\OO\left(z^{-2}\right).
\end{multline*}
Consequently, the function $s(m^{(\kappa)}_N-m)/r=-2f(m^{(\kappa)}_N-m)$ is $\OO(z^{-1})$ as $|z|\to\infty$. Moreover, it is analytic outside $\mathcal{L}$, so
the Cauchy integral formula yields
$$
\frac{1}{2\pi\ii}\int_{\mathcal{L}}\frac{s(\xi)(m^{(\kappa)}_N(\xi)-m(\xi))}{r(\xi)(z-\xi)}\rd\xi=-2f(z)(m_N^{(\kappa)}-m)(z),
$$
proving
\begin{equation}\label{eqn:withoutB}
-2f(z)(m_N^{(\kappa)}(z)-m(z))=
-\frac{1}{2\pi\ii}\int_{\mathcal{L}}
\frac{(m^{(\kappa)}_N(\xi)-m(\xi))^2-c^{(\kappa)}_N(\xi)}{r(\xi)(z-\xi)}\rd\xi+
\OO(N^{-10}).
\end{equation}
Consider now the following rectangular contours, defined by their vertices:
\begin{align}
{\mathcal{L}}_1:\ &\notag
B+3\kappa+\ii 3\kappa,A-3\kappa+\ii 3\kappa, A-3\kappa-\ii 3\kappa, B+3\kappa-\ii 3\kappa,\\
{\mathcal{L}}_2:\ &\label{eqn:contour}
B+4\kappa+\ii 4\kappa,A-4\kappa+\ii 4\kappa, A-4\kappa-
\ii 4\kappa, B+4\kappa-\ii 4\kappa.
\end{align}
In particular, note that all the zeros of $r$ are strictly outside
$\mathcal{L}_2$.
For $z$ inside $\mathcal{L}_2$ and $\Im(z)\geq N^{-1}$, by the Cauchy formula, equation (\ref{eqn:withoutB}) implies that
\begin{multline}\label{eqn:onL1}
-2s(z)(m_N^{(\kappa)}(z)-m(z))\\=-(m^{(\kappa)}_N(z)-m(z))^2+c^{(\kappa)}_N(z)
-\frac{r(z)}{2\pi\ii}\int_{\mathcal{L}_2}\frac{(m^{(\kappa)}_N(\xi)-m(\xi))^2-c^{(\kappa)}_N(\xi)}
{r(\xi)(z-\xi)}\rd\xi+
\OO(N^{-10}).
\end{multline}
In the above expression, if now $z$ is on $\mathcal{L}_1$,
$|z-\xi|\geq\kappa$, and on $\mathcal{L}_2$
$|r|$
is separated away
from zero by a positive universal constant.
Moreover,
$c_N^{(\kappa)}(\xi)$ can be bounded in the following way. For any $\xi\in \mathcal{L}_2$, there is a
smooth function $g_\xi$ supported on $[A-2\kappa,B+2\kappa]$ which coincides with  $\frac{1}{\xi-\lambda_k}$ on $[A-\kappa,B+\kappa]$,
Moreover, this choice can be made such that
$\|g_\xi\|_\infty,\|g'_\xi\|_\infty,\|g''_\xi\|_\infty$ are uniformly bounded in
$\xi\in\mathcal{L}_2$. Then
\begin{multline*}
\frac{1}{N^2}\left|\var_{\mu^{(N,\kappa)}}\left(\sum_{k=1}^N\frac{1}{\xi-\lambda_k}\right)\right|
=
\frac{1}{N^2}\left|\var_{\mu^{(N,\kappa)}}
\left(\sum_{k=1}^Ng_\xi(\lambda_k)\right)\right|\\
=
\frac{1}{N^2}\left|\var_{\mu^{(N)}}
\left(\sum_{k=1}^N g_\xi(\lambda_k)\right)\right|(1+\oo(1)),
\end{multline*}
where the last equality follows from (\ref{eqn:BPS1}).
Now, from Lemma \ref{lem:LinStat}, this last variance is uniformly bounded by $c\,(\log N)^2$,
with $c$ uniformly bounded in $\xi$.
This proves that $k^{(\kappa)}_N(\xi)$ is $\OO((\log N)^2/N^{2})$, uniformly on the contour $\mathcal{L}_2$.
Moreover,
$\frac{1}{N}{m^{(\kappa)}_N}'=\OO(N^{-1})$, so finally $c_N^{(\kappa)}(\xi)$ is
uniformly $\OO(N^{-1})$ on $\mathcal{L}_2$ and  (\ref{eqn:onL1}) implies
$$
-2s(z)(m_N^{(\kappa)}(z)-m(z))=-(m^{(\kappa)}_N(z)-m(z))^2(z)+\OO\left(\sup_{\mathcal{L}_2}|m_N^{(\kappa)}-m|^2\right)+
\OO(N^{-1}).
$$
Moreover, from the maximum principle for analytic functions,
$\sup_{\mathcal{L}_2}|m_N^{(\kappa)}-m|
\leq
\sup_{\mathcal{L}_1}|m_N^{(\kappa)}-m|$, so the previous equation implies
$$
\sup_{\mathcal{L}_1}|m_N^{(\kappa)}-m|=\OO\left(\sup_{\mathcal{L}_1}|m_N^{(\kappa)}-m|^2+\frac{1}{N}\right).
$$
We know that $\rho_1^{(N)}(s)\rd s$ converges weakly to $\rho (s)\rd s$ (see \cite{AndGuiZei2010}),
so by (\ref{eqn:BPS1}) and (\ref{eqn:BPS2}) $\rho_1^{(N,\kappa)}(s)\rd s$ converges weakly to  $\rho (s)\rd s$.
On $\mathcal{L}_1$, $z$ is at distance at least $\kappa$ from the support of both $\rho_1^{(N,\kappa)}(s)\rd s$
and $\rho(s)\rd s$ so, on $\mathcal{L}_1$, $m_N^{(\kappa)}-m$ converges
 uniformly to $0$. Together with the above equation, this implies that
$$
\sup_{\mathcal{L}_1}|m_N^{(\kappa)}-m|=\OO\left(\frac{1}{N}\right).
$$
By the maximum principle the same estimate holds outside $\mathcal{L}_1$,
in particular on $\mathcal{L}_2$, so
equation (\ref{eqn:onL1}) implies that for $z$ inside $\mathcal{L}_1$
\begin{equation}\label{eqn:ReFirstLoop}
-2s(z)(m_N^{(\kappa)}(z)-m(z))=-(m^{(\kappa)}_N(z)-m(z))^2+c^{(\kappa)}_N(z)+\OO\left(\frac{1}{N}\right).
\end{equation}
Moreover,
\begin{multline}\label{eqn:mNPrime}
\frac{1}{N}  | {m_N^{(\kappa)}}'(z)| = \frac{1}{N^2 }  \left |
\E_{\mu^{(N,\kappa)}}    \sum_j  \frac {1}  {(z-\lambda_j)^2} \right |\\
\le  \frac{1}{N \eta  } \Im   \, m^{(\kappa)}_N(z)\leq \frac{1}{N \eta  } |m^{(\kappa)}_N(z)-m(z)|+\frac{1}{N\eta}|\Im\ m(z)|\leq
\frac{1}{N \eta  } |m^{(\kappa)}_N(z)-m(z)|+\frac{c}{N\eta}
\end{multline}
for some constant $c$.
We used the well-known fact that $\Im\ m$ is uniformly bounded on the upper half plane\footnote{This follows for example from properties of the Cauchy operator, see p 183 in \cite{Dei1999}.}.
On the set $A+\delta <E<B-\delta$ and $|\eta|<\kappa$, we have  $\inf |s|>0$. Therefore
(\ref{eqn:ReFirstLoop}) takes the form
\begin{equation}\label{eqn:upBoundSt}
\left(1+\OO\left(\frac{1}{N\eta}\right)\right)(m^{(\kappa)}_N(z)-m(z))=
\OO\left(|m^{(\kappa)}_N(z)-m(z)|^2+\frac{1}{N^2}k^{(\kappa)}_N(z)+\frac{1}{N\eta}\right).
\end{equation}
{F}rom the hypothesis (\ref{eqn:kNTo0}), if $N^{-1+a}\leq \eta\leq\kappa$
and $A+\delta <E<B-\delta$, then
\begin{equation}\label{eqn:quadratic}
|m^{(\kappa)}_N-m|\leq c|m^{(\kappa)}_N-m|^2+\e_N,
\end{equation}
for some $c>0$ and $\e_N\to 0$ as $N\to\infty$.
For large $N$, (\ref{eqn:quadratic}) implies that
$|m^{(\kappa)}_N-m|\leq 2\e_N$ or $|m^{(\kappa)}_N-m|\geq 1/c-2\e_N$.
 Together with $|m^{(\kappa)}_N-m|(E+ \ii\kappa)\to 0$
and the continuity of
$|m^{(\kappa)}_N-m|$ in the upper half plane, this implies that
$|m^{(\kappa)}_N-m|\leq 2\e_N$ and therefore
$|m^{(\kappa)}_N-m|\to 0$ uniformly on $N^{-1+a}\leq \eta\leq\kappa$, $A+\delta <E<B-\delta$.
Consequently, using  (\ref{eqn:upBoundSt}), this proves that
there is a constant $c>0$  such that for any
$\eta\geq N^{-1+a}$,
$A+\delta<E<B-\delta$,
$$
|m^{(\kappa)}_N(z)-m(z)|\leq c\left(\frac{1}{N\eta}+\frac{1}{N^2}k^{(\kappa)}_N(z)\right).
$$
The same conclusion remains when substituting $m_N^{(\kappa)}$ (resp. $k_N^{(\kappa)}$)
by $m_N$ (resp. $k_N$) thanks to (\ref{eqn:BPS1}) and (\ref{eqn:BPS2}).
\end{proof}
\vspace{0.3cm}

To prove rigidity results for $\mu$, the above Lemma \ref{lem:Johansson} will be combined with the following Helffer-Sj\"ostrand
estimate,
already proved in the following form in \cite{BouErdYau2011}.

\begin{lemma}\label{lem:HS} Let $\delta<(B-A)/2$ and $E\in[A+\delta,B-\delta]$ and $0<\eta<\delta/2$.
Define a function $f= f_{E,\eta}$: $\R\to \R$
 such that $f(x) = 1$ for $x\in (-\infty, E-\eta]$, $f(x)$ vanishes
for  $x\in [E+\eta, \infty)$, moreover
 $|f'(x)|\leq c\eta^{-1}$ and $|f''(x)|\leq c\eta^{-2}$, for some constant $c$.
 Let $\wt\rho$ be an arbitrary signed measure
and let $S(z)= \int (z-x)^{-1}\wt\rho(x)\rd x$ be its Stieltjes transform.
Assume that, for any $x\in[A+\delta/2,B-\delta/2]$,
\begin{equation}\label{eqn:cond1}
\left| S(x+\ii y)\right|\leq \frac{ U}{Ny}\;\;  \mbox{for}\;\; \eta <y<1 ,\;\;\mbox{and}\;\;
|\Im\, S(x+\ii y)|\leq \frac{ U}{Ny}\;\; \mbox{for}\;\; 0<y<\eta.
\end{equation}
 Assume moreover that $\int_\RR\wt\rho(\lambda)\rd\lambda=0$ and that there is a real constant $\mathcal{T}$ such that  
\begin{equation}\label{eqn:cond2}
\int_{[-\mathcal{T},\mathcal{T}]^c}  |\lambda\wt\rho(\lambda)|\rd\lambda   \le
\frac{U}{N}.
\end{equation}
Then for some  constant $C>0$, independent of $N$ and $E\in [A+\delta,B-\delta]$, we have
$$
\left|\int f_E(\lambda)\wt\rho(\lambda)\rd\lambda \right|  \le
\frac{C U|\log\eta| }{N}.
$$
\end{lemma}

\section{Convexification}\label{sec:convexification}

\subsection{Outline of the main ideas}

The Hamiltonian $\cH=\cH_N$ of the measure $\mu \sim \exp(-\beta N\cH)$ is given by
$$
  \cH = \frac{1}{2}\sum_{k=1}^N V(\lambda_k)
 -\frac{1}{N}\sum_{1\le i<j\le N} \log (\lambda_j-\lambda_i).
$$
$\cH$ is not convex, but its Hessian is bounded from below, $\nabla^2 \cH\ge - W$.
We will modify this Hamiltonian by an additional term
\begin{equation}\label{eqn:Htilda}
  \wt\cH := \cH + (W+1)\sum_{\al=1}^\ell X_\al^2, \qquad X_\al = N^{-1/2}\sum_{j=1}^N \left(g_\al(\lambda_j)
  - g_\al(\tgamma_j)\right),
\end{equation}
where the real functions $g_\al$, $\alpha=1,2,\ldots, \ell$, will be determined later
and will be independent of $N$.  Here we denoted by $\tgamma_j$ a slightly modified
version of the classical location of the points, defined by the relation
\begin{equation}\label{tgamma}
   \int_{A}^{\tgamma_j} \rho(s)\rd s = \frac{j-\frac{1}{2}}{N}, \qquad j=1,2,\ldots , N.
\end{equation}
Compared with $\gamma_j$ defined in \eqref{gammadef}, there is a small shift in the definition
which makes a technical step (Lemma \ref{lem:opIneq})
easier in this section. In all estimates involving $\gamma_j$ this small shift plays no
role since $\max_j |\gamma_j-\tgamma_j|\le CN^{-2/3}$.
In particular the crude large deviation bound \eqref{eqn:largDev1} holds  for $\tgamma$'s as well:
\begin{align}
\P_\mu\left( \exists k\in\llbracket1,N\rrbracket\mid
 | \lambda_k-  \tgamma_k| \ge \e \right)\leq c_1 e^{-c_2 N}. \label{eqn:largDev1new}
\end{align}

The $N^{-1/2}$ normalization in the definition of $X_\alpha$ is chosen such that the vector
$$
 \bG_\al: = N^{-1/2}\big(g_\al'(\tgamma_1), g_\al'(\tgamma_2), \ldots, g_\al'(\tgamma_N)\big)\in \RR^N
$$
is $\ell^2$-normalized.

Define the random variables
$$
\Delta : = \max\left(\frac{1}{N}\sum_j|\lambda_j-\tgamma_j|,\frac{1}{N}\sum_j(\lambda_j-\tgamma_j)^2\right)\leq \Delta^{(\delta)}:=
\delta+\frac{1}{N\delta}\sum_j(\lambda_j-\tgamma_j)^2
$$
for any $0<\delta<1$. Clearly
$$
|X_\al|\le N^{1/2}\|g'_\al\|_\infty \Delta.
$$
We then have, for any vector $\bv\in \RR^N$, that
\begin{align}\notag
  \langle \bv, (\nabla^2\wt\cH) \bv\rangle  &
= \frac{1}{N}\sum_{i<j} \frac{(v_i-v_j)^2}{(\lambda_i-\lambda_j)^2}
  + \frac{1}{2} \sum_j V''(\lambda_j)v_j^2\\& + 2(W+1)\sum_{\al=1}^\ell \Bigg[
 \Big( \frac{1}{\sqrt{N}}\sum_j g_\al'(\lambda_j)v_j\Big)^2 +
X_\al\sum_j \frac{1}{\sqrt{N}}g_\alpha''(\lambda_j)v_j^2 \Bigg] \notag \\
&\ge  \frac{1}{N}\sum_{i<j} \frac{(v_i-v_j)^2}{(\lambda_i-\lambda_j)^2}
  +(W+1)\sum_{\al=1}^\ell |\langle \bG_\al, \bv\rangle|^2 \notag \\
  &- \Big[ W + 2\Delta  (W+1)\sum_{\al=1}^\ell  \big(\| g''_\al\|_\infty^2
+\|g_\al'\|_\infty \| g''_\al\|_\infty\big)\Big] \|\bv\|^2
\label{hessest}
\end{align}
where we used a  simple Schwarz inequality
\begin{align*}
   2 \Big( \sum_j g_\al'(\lambda_j)v_j\Big)^2 & \ge  \Big( \sum_j g_\al'(\tgamma_j)v_j\Big)^2
  - 2 \Big( \sum_j \big[  g_\al'(\lambda_j)- g_\al'(\tgamma_j)\big]v_j\Big)^2 \nonumber
\\
 & \ge    \Big( \sum_j g_\al'(\tgamma_j)v_j\Big)^2
  - 2  \sum_j \big[  g_\al'(\lambda_j)- g_\al'(\tgamma_j)\big]^2 \|\bv\|^2
\nonumber
\end{align*}
in the last step.

We will define below a nonnegative symmetric operator $\cQ$ on $\CC^N$
via a quadratic form
\begin{equation}\label{eqn:Q}
   \langle \bv, \cQ\bv\rangle = \sum_{i,j=1}^N Q_{ij} (v_i-v_j)^2
\end{equation}
such that for typical point configuration $\bla=(\lambda_1, \lambda_2,
\ldots , \lambda_N)$ we have
\begin{equation}\label{tp}
   \frac{1}{N}\sum_{i<j} \frac{(v_i-v_j)^2}{(\lambda_i-\lambda_j)^2} \ge \sum_{i,j=1}^N
Q_{ij} (v_i-v_j)^2.
\end{equation}
In our applications, we will then choose $\ell$ to be a large but $N$-independent number,
we  will let $\bG_\al$, $\al=1,2,\ldots, \ell$ be the eigenfunctions corresponding to
 the lowest $\ell$ eigenvalues $\mu_1\le \mu_2\le \ldots \le \mu_\ell$ of the
nonnegative operator $\cQ$.
 Thus we will have the operator inequality
$$
   \cQ+(W+1) \sum_{\al=1}^\ell |\bG_\al\rangle\langle \bG_\al| \ge \min\{ (W+1),\mu_{\ell+1}\}.
$$
In Section~\ref{sec:slow} we will show that
for $\ell$ sufficiently large, independent of $N$, we have $\mu_{\ell+1}> W+1$.
Setting
$$
  C(\ell): =2(W+1)\sum_{\al=1}^\ell  \big(\| g''_\al\|_\infty^2
+\|g_\al'\|_\infty \| g''_\al\|_\infty\big),
$$
we will obtain from \eqref{hessest} that
\be
   \langle \bv, (\nabla^2\wt\cH) \bv\rangle \ge \Big((W+1)
 - W - C(\ell)\Delta^{(\delta)}\Big)\|\bv\|^2 \ge 0
\label{Htildeconv}
\ee
as long as $\Delta^{(\delta)} \le 1/C(\ell)$. From now,
 we choose $\delta=1/(2 C(\ell))$ and as $N\to\infty$, we have
$\Delta^{(\delta)}\le 1/C(\ell)$ with very high probability,
thanks to the large deviation estimates (\ref{eqn:largDev1new}).

To prove that $\mu_{\ell+1}\ge W+1$, we only need to estimate the low lying eigenvalues of $\cQ$
and  we need to understand  the low lying eigenfunctions $\bG_\alpha$.
 Since the only requirement for $Q_{i,j}$ is to satisfy the bound \eqref{tp},
we have  a substantial freedom in choosing $Q_{i,j}$ conveniently.  There are many ways to choose $Q_{i,j}$;
 we will give one possible approach  that relies on
enlarging the space by  a  reflection principle in the next section.
 Roughly speaking, we will construct
an  operator $\cR$ with periodic boundary conditions
on the  set consisting of the original set and  its ``reflection".
 We then choose   $\cQ$ to be  the restriction of  $\cR$ to  the symmetric
(under the reflection)  sector.
The  operator $\cR$ is  translation
invariant, hence it can be diagonalized via Fourier transform and the eigenfunctions are explicit.
The reader may skip the next section on first reading as it contains
fairly elementary arguments that are independent of the rest of the paper.

\subsection{Slow modes analysis}\label{sec:slow}

 Let $I:=\llbracket 1,N\rrbracket$
be the index set of the vectors $\bv$.
The original operator $\cQ$ is defined on
the space $\ell^2(I)$.  We enlarge this space to
$\ell^2(\wt I)$, where $\wt I:=\llbracket -N+1 ,N\rrbracket$.
We extend any vector  $\bv\in \ell^2(I)$ by reflection to
a vector $\wt\bv\in \ell^2(\wt I)$ as follows
\begin{align}
   \wt v_j = & v_j, \qquad j=1,2,\ldots, N;
\label{wtv}
\\
  \wt v_j = & v_{1-j},\quad j=0,\ldots, -N+1.
\end{align}
We will often view the set $\wt I$
modulo $2N$ periodic, i.e. we consider it
as $2N$ points on a circle and identify
$-N$ with $N$.  We can thus
 also view $\ell^2(\wt I)$ as the space of vectors
with periodic boundary condition $\wt v_{-N}=\wt v_N$.
The algebraic operations on the indices will be
considered modulo $2N$.

We consider the natural translation invariant distance on $\wt I$.
Define the function $m(n)$ for $n\in \ZZ$ such that $m(n)\in \llbracket -N+1, N\rrbracket$
and $m(n) \equiv n  \; \mbox{mod} (2N)$.
Then the distance  between $k, \ell\in \wt I$ is
defined as $d(k,\ell):= |m(k-\ell)|$ which ranges from 0 to $N$.

\begin{lemma}\label{lm:SQ} Let $\e>0$ be sufficiently small, depending only on $V$. Define
\be\label{Rdef}
{ R^{(\e)}_{k,\ell}= } R_{k,\ell}:=\frac{1}{N}\frac{\e^{2/3}}{\frac{d(k,\ell)^2}{N^2}+\e^2},
\qquad k,\ell\in \wt I =\llbracket -N+1 ,N\rrbracket,
\ee
and
\be\label{Qdef}
{ Q^{(\e)}_{i,j}=} Q_{i,j }:=R_{i,j}+R_{1-i,j}+R_{i,1-j}+R_{1-i,1-j},\qquad i,j\in  I=\llbracket 1 ,N\rrbracket.
\ee
Then there is a constant $c_1>0$ depending only on $V$ such that for any $\e>0$
there is a constant $c_2>0$ (depending on $V$ and $\e$) such that for any
$N$ and $i,j\in\llbracket 1,N\rrbracket$
$$
\mathbb{P}_\mu\left(\frac{1}{N}\frac{1}{(\lambda_i-\lambda_j)^2}\leq c_1 \ Q_{i,j}\right)\leq e^{-c_2 N}.
$$
\end{lemma}
The relation between $Q_{i,j}$ and $R_{i,j}$ is dictated by the requirement that
$$
  \langle \bv, \cQ\bv\rangle_{\ell^2(I)} = \langle\tilde \bv, \cR \tilde \bv\rangle_{\ell^2(\wt I)}
: = \sum_{i,j\in \wt I} R_{i,j}(\tilde v_i-\tilde v_j)^2,
  \qquad \forall\,\bv\in \ell^2(I),
$$
which can be easily checked from \eqref{wtv}.

\vspace{.2cm}

\begin{proof}
Recall that $[A,B]$ is the support of $\rho$, $\rho>0$ on $(A,B)$
and  $\rho$
has a square-root singularity  at the two endpoints, i.e.
it vanishes  as $\rho(x)\sim s_A\sqrt{x-A}$
as $x\to A^+$ and
$\rho(x)\sim s_B\sqrt{B-x}$
as $x\to B^-$ with some positive $s_A, s_B$.

From the large deviations of the extreme eigenvalues (included in (\ref{eqn:largDev1})), we know that
for any $\kappa>0$ there is a $c(\kappa)>0$ such that
\be\label{extreme}
  \mathbb{P}_\mu\Big( \lambda_1\le A-\kappa\Big)+
 \mathbb{P}_\mu\Big( \lambda_N\ge B+\kappa\Big)\le e^{-c(\kappa)N}.
\ee

Fix a positive number $s< \min(s_A,s_B)$. Then there is a $\kappa_0>0$,
depending only on $V$, such that
\be
   \rho(x)\ge s\sqrt{x-A} \cdot \mathds{1}_{x\in [A, A+\kappa_0]} + s\sqrt{\kappa_0}
   \cdot \mathds{1}_{x\in [A+\kappa_0, B-\kappa_0]}+
   s\sqrt{B-x} \cdot \mathds{1}_{x\in [B-\kappa_0, A]}.
\label{lowerrho}
\ee
Let $\e\le c \kappa_0^{3/2}$ with a small positive constant $c$.
Suppose that $k\le N/2$; if $k$ is near the upper edge, the argument is similar.
Since
$$
  \int_{-\infty}^{\tgamma_k}\rho = \frac{k-\frac{1}{2}}{N} =  \int_{-\infty}^{\lambda_k} \frac{1}{N}\sum_m
  \delta_{\lambda_m} - \frac{1}{2N},
$$
from the first relation
we get
\be
c (k/N)^{2/3}\le \tgamma_k -A\le C  (k/N)^{2/3}
\label{gammaloc}
\ee
 with some positive constants
$c, C$, depending only on $V$. Subtracting the first and second relations and
using (\ref{eqn:largDev1new}), we obtain that for any fixed $K$
\be
    \Big|\int_{\tgamma_k}^{\lambda_k}\rho \Big|\le \frac{\e}{K}
\label{rhosm}
\ee
apart from an event
of exponentially small probability (i.e. of type $\exp(-c(\e/K)N)$).

Additionally, assume now that $k\ge N\e$.
Under \eqref{rhosm} we easily see that  $\lambda_k \in (\tgamma_{k/2}, \tgamma_{3k/2})$,
since both $\int_{\tgamma_{k/2}}^{\tgamma_k}\rho$ and  $\int_{\tgamma_{k}}^{\tgamma_{3k/2}}\rho$
are of the order $k/N$ which is larger than $\e/K$ if $K$ is large enough
(depending only on $V$).
Then \eqref{lowerrho} and \eqref{rhosm} imply
$$
   |\lambda_k-\tgamma_k|\le C\e (\tgamma_k-A)^{-1/2} \le C\e(k/N)^{-1/3}, \qquad N\e\le k\le N/2
$$
with exponentially high probability and
with a constant $C$ depending only on $V$.

Now we consider the $k\le N\e$ case. Using \eqref{extreme} with
$\kappa = \e^{2/3}$ and \eqref{gammaloc}, we have (apart from an event
of exponentially small probability)
$$
    \tgamma_k-\lambda_k\le C(k/N)^{2/3}+\kappa  \le C\e^{2/3}.
$$
Finally, still when $k\le N\e$, i.e. $\tgamma_k\le A+C\e^{2/3}$
then \eqref{rhosm} implies that
 $\lambda_k\le A+C_1\e^{2/3}$ with a large $C_1$, i.e.
$$
   \lambda_k-\tgamma_k \le C\e^{2/3}, \qquad k\le N\e,
$$
still apart from an event
of exponentially small probability.
Summarizing all cases, we obtain that
\be
   |\lambda_k-\tgamma_k|\le \frac{C\e}{(k/N)^{1/3}+\e^{1/3}}, \qquad k\le N/2.
\label{gammaminuslambda}
\ee
holds with overwhelming probability.

Now let $|i-j|\ge N\e$, then
$$
    \frac{|i-j|}{N} =\Big|\int_{\tgamma_i}^{\tgamma_j}\rho\Big|
   \ge  c|\tgamma_i-\tgamma_j|\e^{1/3}
$$
since either $i$ or $j$ is larger than $N\e$ and smaller than $N(1-\e)$,
say $N\e \le i\le N(1-\e)$, and then
$\rho$ is at least of order $\e^{1/3}$ in the neighborhood
of $\tgamma_i$. If $|i-j|\le N\e$, then we have the
trivial bound  $|\tgamma_i-\tgamma_j|\le C\e^{2/3}$. Combining these,
$$
    |\tgamma_i-\tgamma_j|\le \frac{C|i-j|}{N\e^{1/3}} +  C\e^{2/3}
$$
holds for any $i,j$.
Furthermore, clearly $|\lambda_i-\tgamma_i|\le C\e^{2/3}$ from \eqref{gammaminuslambda},
so we have proved that
$$
   |\lambda_i-\lambda_j|\le \frac{C|i-j|}{N\e^{1/3}} + C\e^{2/3}
$$
with overwhelming probability and for any $i,j$. In other words,
there is a constant $C$ (depending only on $V$) such that
for any sufficiently small $\e$ and for some $c(\e)>0$
we have
for any $N$ and $i,j\in\llbracket 1,N\rrbracket$
\begin{equation}\label{eqn:nonCirc}
\mathbb{P}\left(\frac{1}{(\lambda_i-\lambda_j)^2}<
\frac{C\e^{2/3}}{\e^2+\frac{|i-j|^2}{N^2}}\right)
<e^{-c(\e) N}.
\end{equation}

\bigskip

The proof of Lemma \ref{lm:SQ} will therefore be complete if we can prove that
$$
|i-j|\leq \min\big\{ d(i,j), d(1-i,j), d(i, 1-j), d(1-i, 1-j)\big\}, \qquad i,j\in I.
$$
Note that we can assume $i>j$ and then $|i-j|= |m(i-j)|=d(i,j)$ is obvious.
Moreover, as $d$ is symmetric to the reflection $i\to 1-i$ on $\wt I$, i.e.
$d(k,\ell)=d(1-k, 1-\ell)$, we just need to prove that
$|i-j|\leq d(1-i,j)$.
 If $i+j\leq N+1$,  then
$d(1-i, j)=i+j-1>|i-j|$.
If $N+2 \le i+j\le 2N$, then $d(1-i,j) = |m(i+j-1)| = |i+j-1-2N| = 2N+1-i-j> i-j$
because $i\leq N$, completing the proof.
\end{proof}\vspace{0.3cm}

We use the matrix $Q$ in the previous lemma instead of bounds of type $(\ref{eqn:nonCirc})$
because it is related to $R$, a circulant matrix, allowing to derive its eigenvalues
and eigenvectors in an explicit way.

\begin{lemma}\label{lem:S}
{ Let $R=R^{(\e)}$ be the  matrix  $(R_{i,j})_{(i,j)\in{\tilde I}^2}$,
where the matrix elements $R_{i,j}$ are defined in \eqref{Rdef}.}
Then the eigenvalues $\nu_1,\dots,\nu_{2N}$ of $R$ are
$$
\nu_k=\sum_{j=0}^{2N-1} e^{\ii 2\pi j\frac{k}{2N}}R_{0,j}
$$
and the associated normalized eigenvector is $u_k^*=(2N)^{-1/2}(e^{\ii 2\pi j \frac{k}{2N}})_{j=-N+1,\dots,N}$.

In particular, for any given $W>0$ there is a sufficiently small $\e$ such that
 for large enough $N$ we have $\nu_{2N}>W$.
Moreover,
for any given $\e>0$ and $s>0$
there is some $a>0$ depending only on $\e$ and $s$
such that for any $N$
$$\left\{k: |\nu_k|>s\right\}\subset\llbracket 1,
a\rrbracket\cup\llbracket 2N+1-a,2N\rrbracket.
$$
\end{lemma}

We remark that the matrix $R$  defines a symmetric operator $\bw \to R\bw$ and a
 quadratic form  $\langle  \bw, R\bw\rangle
= \sum R_{i,j}w_iw_j$ in $\ell^2(\wt I)$. It is related to the quadratic form $\cR$
via
\begin{equation}\label{RR}
   \langle \bw, \cR \bw\rangle = \sum_{i,j} R_{i,j}(w_i^2 + w_j^2)
  -2\sum_{i,j} R_{i,j} w_iw_j = 2 \nu_{2N}\|\bw\|^2 - 2 \langle \bw, R\bw\rangle,
\end{equation}
i.e. we have $\cR=2\nu_{2N}- 2R$ on $\ell^2(\wt I)$.

\medskip

\begin{proof}
The first assertions, about the eigenvalues and eigenvectors, is a general fact about circulant matrices
 and can be obtained by
 Fourier transform in $\left\{0,\frac{1}{2N},\dots,\frac{2N-1}{2N}\right\}$.

Concerning the distribution of eigenvalues, note that
$$
\nu_{2N}=\frac{1}{N}\sum_{j=1}^{2N-1} \frac{\e^{2/3}}{\frac{m(j)^2}{N^2}+\e^2}
\underset{N\to\infty}{\longrightarrow} \e^{2/3}\int_{-1}^1\frac{\rd x}{x^2+\e^2}
=
\e^{-1/3}\int_{-1/\e}^{1/\e}\frac{\rd x}{x^2+1}.
$$
We therefore have, for sufficiently small $\e>0$, $\nu_{2N}>W$ for large enough $N$.
We now write
$$
 a_j=\sum_{\ell=0}^{j}e^{\ii 2\pi \ell\frac{k}{2N}},\qquad
b_j=\frac{1}{N} \frac{\e^{2/3}}{\frac{m(j)^2}{N^2}+\e^2}, \qquad 0\leq j\leq 2N-1,
$$
 with the convention $b_{2N}=0$.
Then note that $|a_j|<2|1-e^{\ii 2\pi\frac{k}{2N}}|^{-1}$, and write $\nu_k=\sum_{j=0}^{2N-1}a_j(b_j-b_{j+1})$
to obtain
$$
|\nu_k|\leq 4\e^{2/3}\frac{1}{N|1-e^{\ii 2\pi\frac{k}{2N}}|}\
\sum_{j=0}^{N}\left|
\frac{1}{\e^2+\left(\frac{j}{N}\right)^2}
-\frac{1}{\e^2+\left(\frac{j+1}{N}\right)^2}\right|\leq
\frac{c(\e)}{N|1-e^{\ii 2\pi\frac{k}{2N}}|}.
$$
Consequently, if $|\nu_k|>s$ then $|1-e^{\ii 2\pi\frac{k}{2N}}|< c(\e,s)/N$,
which only happens for indices $k$ whose distance to $0$ or $2N$ is uniformly bounded.
\end{proof}

\begin{lemma}\label{lem:opIneq}
For a fixed $\al\in \NN$,
there exists  an $N$-independent function  $g_\alpha$ with
$\|g_\alpha\|_\infty+\|g_\alpha'\|_\infty+\|g_\alpha''\|_\infty<\infty$ such that
$$
g_\alpha'(\tgamma_k)=\sqrt{2}\cos\left(2\pi (k-\frac{1}{2})\frac{\alpha}{2N}\right).
$$
 Define
$$
\bG_\al=N^{-1/2}(g_\alpha'(\tgamma_1),\dots,g_\alpha'(\tgamma_N))=
\sqrt{2}N^{-1/2}\left(\cos \left(\pi (2j-1) \frac{\alpha}{2N}\right)\right)_{j=1,\dots,N},
$$
it is easy to see that  $\| \bG_\al\|=1$.
For any $M>0$ there is $\e>0$ and $\ell>0$ such that, for large enough $N$ we have, as an inequality between
 positive operators on $\ell^2(I)$,
$$
\cQ+M \sum_{\al=1}^\ell |\bG_\al\rangle\langle \bG_\al| \ge M,
$$
where { $\cQ=\cQ^{(\e)}$ was defined in (\ref{eqn:Q}) with coefficients $Q_{i,j}=Q_{i,j}^{(\e)}$ defined
in \eqref{Qdef}. }
\end{lemma}

\begin{proof}   The existence of $g_\al$ follows easily from the
fact that the density $\rho$ has a square-root
singularity near the edges, see \eqref{sqsing}. In fact, we can choose
the functions $g_\al$ such that
  $g_\al'(x) =\sqrt{2}\cos \big(\frac{2}{3}\pi s_A \alpha (x-A)_+^{3/2}\big)$ for
$x$ near $A$,
and a similar relation holds near the upper edge.
As $\langle \tilde \bv, \cR \tilde \bv\rangle=\langle \bv, \cQ \bv\rangle$,
 $\| \tilde \bv\|^2 =2\|\bv\|^2$
and
$$
   |\langle \bv, \bG_\al\rangle|^2 = |\langle \tilde \bv, u_\al\rangle|^2 = |\langle
\tilde\bv, u_{2N-\al}\rangle|^2,
$$
 we just need to prove that the
 operator inequality
$$
2\cR+M \sum_{k\in\llbracket 1,
\ell\rrbracket\cup\llbracket 2N-\ell,2N-1\rrbracket} |u_k\rangle\langle u_k| \ge M
$$
holds in $\ell^2(\wt I)$
 for some fixed constant $\ell$
and for large enough $N$.
 This is equivalent to the statement
\begin{equation}\label{eqn:SIneq}
R+\frac{M}{4}\leq \nu_{2N}+\frac{M}{4}\sum_{k\in\llbracket 1,
\ell\rrbracket\cup\llbracket 2N-\ell,2N-1\rrbracket} |u_k\rangle\langle u_k|,
\end{equation}
since $\cR=2\nu_{2N}- 2R$, see \eqref{RR}.
{ Recalling that $\nu_k$ are the eigenvalues of $R=R^{(\e)}$, we
need to check that for some  fixed $\e>0$ and $\ell$ we have
$$
\nu_k+\frac{M}{4}\leq \nu_{2N}+\frac{M}{4}\mathds{1}_{k\in\llbracket 1,
\ell\rrbracket\cup\llbracket 2N-\ell,2N-1\rrbracket}
$$
 for any $k$. Since $\nu_{2N}$ is the top eigenvalue,
this inequality is obvious if $k\in\llbracket 1,
\ell\rrbracket\cup\llbracket 2N-\ell,2N-1\rrbracket$.}
Moreover,
Lemma \ref{lem:S} proves the existence of { some fixed $\e>0$ and  $\ell$,
such that $\nu_{2N}>M$ and $\{k: |\nu_k|\leq \frac{3}{4} M\}\subset
\llbracket 1,
\ell\rrbracket\cup\llbracket 2N-\ell,2N-1\rrbracket$ hold for large enough $N$ (depending
on $M$ as well as $\e$ and $\ell$). This concludes the proof. }
\end{proof}

\subsection{The locally constrained measures}

In this section some arbitrary $\e,\alpha>0$ are fixed.
Let $\theta$ be a continuous nonnegative function with $\theta=0$ on $[-1,1]$ and $\theta''\geq 1$ for $|x|>1$.
We can take for example $\theta(x)=(x-1)^2 \mathds{1}_{x>1}+(x+1)^2 \mathds{1}_{x<-1}$ in the following.

\begin{definition} \label{def:locallyConstrained}
For any $ s, \ell > 0$,
define the probability measure
\begin{equation}\label{eqn:omega}
\rd\nu^{(s, \ell)}=e^{-\beta N\cH_\nu}:=\frac{1}{Z^{(s,\ell) }}
e^{-\beta N \psi^{(s)}-\beta N\sum_{i,j}\psi_{i,j} -\beta  N (W+1) \sum_{\alpha= 1 } ^\ell  X_\alpha^2 }\rd\mu,
\end{equation}
where
\begin{itemize}
\item the function $g_\alpha$ was defined in Lemma $\ref{lem:opIneq}$;
\item $X_\alpha = N^{-1/2} \sum_j  \left( g_\alpha(\lambda_j) -  g_\alpha(\tgamma_j)\right)$;
\item $\psi^{(s)}(\lambda)=N\theta\left(\frac{s}{N}\sum_{i=1}^N(\lambda_i-\tgamma_i)^2\right)$;
\item $\psi_{i,j}(\lambda)=\frac{1}{N}\theta\left(\sqrt{c_1\, N\, Q_{i,j}}(\lambda_i-\lambda_j)\right)$, where $c_1$ was defined in Lemma \ref{lm:SQ}.
\end{itemize}
\end{definition}

\begin{lemma}\label{lem:convex}
There are constants $c,\ell,s>0$ depending only on $V$, such that for $N$ large enough
$\nu^{(s,\ell)}$ satisfies, for any $\bv\in\RR^N$,
$$
\langle \bv,(\nabla^2 \cH_\nu) \bv\rangle\geq c \|\bv\|^2.
$$
\end{lemma}

\begin{proof}
Using the notation (\ref{eqn:Htilda}), we have  $\cH_\nu=\tilde\cH+\psi^{(s)}+\sum_{i,j}\psi_{i,j}$ up to an additive constant, so
\begin{align}
\langle \bv,(\nabla^2 \cH_\nu )\bv\rangle=&\notag
\langle \bv(\nabla^2 \tilde\cH ),\bv\rangle+c_1\sum_{i,j}Q_{i,j}\theta''\left(\sqrt{c_1\,N\, Q_{i,j}}(\lambda_i-\lambda_j)\right)(v_i-v_j)^2\\
+&
\frac{(2s)^2}{N}\label{eqn:conv1}
\theta''\left(\frac{s}{N}\sum_{i=1}^N(\lambda_i-\tgamma_i)^2\right)
\left(\sum_{i=1}^N(\lambda_i-\tgamma_i)v_i\right)^2
+
2s\,
\theta'\left(\frac{s}{N}\sum_{i=1}^N(\lambda_i-\tgamma_i)^2\right)\|\bv\|^2.
\end{align}
We now use (\ref{hessest}) to get a lower bound for $\langle \bv(\nabla^2 \tilde\cH ),\bv\rangle$. Note that
the second $\theta''$ term is positive and that for $x>0$, $\theta'(x)\geq2(x-1)\mathds{1}_{x>1}$; we therefore
get the following lower estimate of (\ref{eqn:conv1}):
\begin{multline*}
\frac{1}{N}\sum_{i,j} \left(\frac{1}{(\lambda_i-\lambda_j)^2}+c_1\,N\,Q_{i,j}
\mathds{1}\Big(\frac{1}{(\lambda_i-\lambda_j)^2}<c_1\,N\, Q_{i,j}\Big)
\right)(v_i-v_j)^2+  (W+1)\sum_{\al=1}^\ell |\langle\bG_\al, \bv\rangle|^2\notag\\
-\left(W + C(\ell)\Delta^{(\delta)}\right)\|\bv\|^2+
4s\,(s\Delta-1)\mathds{1}_{s\Delta-1>0}\|\bv\|^2,
\end{multline*}
which is greater than
$$
\Big\langle \bv,(c_1 \cQ+  (W+1)\sum_{\al=1}^\ell |\bG_\al\rangle\langle \bG_\al|-(W+1))\bv\Big\rangle
+(1-C(\ell)\Delta^{(\delta)})\|\bv\|^2+4s
(s\Delta-1)\mathds{1}_{s\Delta-1>0}\|\bv\|^2.
$$
 Choosing $M=(W+1)/c_1$ in   Lemma \ref{lem:opIneq},  for $\e$ small enough and $\ell$ large enough the above scalar product term is positive for any
$\bv$ and large enough $N$ (note that $c_1$
does not depend on $\e$).
For this choice of $\ell$, taking $s=C(\ell)$ makes the
other terms all together positive (without loss of generality we can assume $C(\ell)>1/4$), concluding the proof.
\end{proof}

From now, we abbreviate $\nu$ for $\nu^{(s,\ell)}$, where $s$ and $\ell$ are fixed such that the conclusion of Lemma \ref{lem:convex} holds.

\subsection{Equivalence of the measures $\nu$ and $\mu$}

We say that a sequence of events $(A_N)_{N\geq 1}$ is exponentially small for
 a sequence of probability measures $(m_N)_{N\geq 1}$
if there are constants $\delta,c_1,c_2>0$ such that for any $N$ we have
$$
m_N\left(A_N\right)\leq c_1 e^{-c_2N^\delta}.
$$

\begin{lemma}\label{lem:equivalence}
For any $s,\ell>0$,
the measures $(\mu^{(N)})_{N\geq 1}$ and  $(\nu^{(s,\ell,N)})_{N\geq 1}$ have the same exponentially small events.
\end{lemma}
\begin{proof}
First note that $\cH_\nu\geq \cH_\mu$, so $Z_\nu\leq Z_\mu$.  We claim that the following inequality holds:
\[
\log Z_\mu\leq\log Z_\nu+ \E_\mu(\beta N (\cH_\nu-\cH_\mu)) .
\]
To prove it, by
Jensen's inequality we have
$$
\log\int e^{\beta N(\cH_\mu-\cH_\nu)}\frac{e^{-\beta N \cH_\mu}}{\int e^{-\beta N \cH_\mu}\rd \lambda}\rd\lambda
\geq
\int \beta N(\cH_\mu-\cH_\nu)\frac{e^{-\beta N \cH_\mu}}{\int e^{-\beta N \cH_\mu}\rd \lambda}\rd\lambda .
$$
We now  bound
\be
\E_\mu( N (\cH_\nu-\cH_\mu)) = \E_\mu(N X_\alpha^2) + \E_\mu(N\psi_{i,j}) + \E_\mu(N\psi^{(s)})
\ee
in the following way.
 By Lemma
\ref{lem:LinStat},
$\E_\mu(N X_\alpha^2)<c_\alpha (\log N)^2$;
by Lemma \ref{lm:SQ} (together to (\ref{eqn:BPS2})),  $\E_\mu(N\psi_{i,j})$ is subexponentially small  for any indices $i$ and $j$;
finally $\E_\mu(N\psi^{(s)})$ is also subexponentially small   by (\ref{eqn:largDev1new})
(together to (\ref{eqn:BPS2})).
Altogether, we get that there is a constant $c>0$ such that for any $N\geq 2$
\begin{equation}\label{eqn:partFunctions}
\log Z_\nu\leq \log Z_\mu\leq \log Z_\nu+ c(\log N)^2.
\end{equation}

Let $(A_N)_{N\geq 1}$ be now a sequence of events exponentially small for $\mu$. By
(\ref{eqn:partFunctions}) we have
$$
\P_\nu(A_N)\leq e^{c(\log N)^2}\P_\mu(A_N),
$$
so $(A_N)_{N\geq 1}$ is also exponentially small for $\nu$.

Assume now that $(A_N)_{N\geq 1}$ is exponentially small for $\nu$:
there are constants $\delta,c_1,c_2>0$ such that for any $N$ we have
$$
\P_\nu\left(A_N\right)\leq c_1 e^{-c_2N^\delta}.
$$
Then for any $t$ we have
\begin{align*}
\P_\mu(A_N)&=\P_\mu(A_N\cap\{\beta N(\cH_\nu-\cH_\mu)>t\})+\P_\mu(A_N\cap\{\beta N(\cH_\nu-\cH_\mu)<t\})\\
&\leq
\P_\mu(\{\beta N(\cH_\nu-\cH_\mu)>t\})+e^{t}\P_\nu(A_N),
\end{align*}
where we used $Z_\nu<Z_\mu$.
Choosing $t=N^{\delta/2}$ makes the second term exponentially
small, and the first one as well by using as previously Lemma \ref{lem:LinStat}, Lemma \ref{lm:SQ}
 and (\ref{eqn:largDev1new}).
\end{proof}
\vspace{0.2cm}

{F}rom the previously proved equivalence of the measures $\mu$ and $\nu$, we can easily obtain
rigidity of the particles at scale $N^{-1/2}$,

\begin{proposition}\label{prop:initial}
For any $\alpha,\e>0$, there are constants
$\delta,c_1,c_2>0$ such that for any $N\geq 1$ and $k\in\llbracket \alpha N,(1-\alpha) N\rrbracket$,
$$
\P_\mu\left(|\lambda_k-\gamma_k|> N^{-\frac{1}{2}+\e}\right)\leq c_1e^{-c_2N^\delta}.
$$
\end{proposition}

\begin{proof}
{F}rom Lemma \ref{lem:convex} about the
convexity of $\mathcal{H}_\nu$, we get by the classical Bakry-\'Emery criterion
\cite{BakEme1983}
that $\nu$ satisfies a logarithmic Sobolev inequality with constant of order $1/N$, so by Herbst's lemma
concentration at scale $N^{-1/2}$ holds for individual particles for $\nu$:
there is a constant $c>0$ such that for any $N\geq 1$, $k\in\llbracket 1,N\rrbracket$ and $x>0$,
$$
\P_\nu\left(|\lambda_k-\E_\nu(\lambda_k)|> x N^{-\frac{1}{2}}\right)\leq e^{-c x}.
$$
 By Lemma \ref{lem:equivalence}, this implies that for some constants $\delta,c_1,c_2>0$,
\begin{equation}\label{eqn:concmu}
\P_\mu\left(|\lambda_k-\E_\nu(\lambda_k)|> N^{-\frac{1}{2}+\e}\right)\leq c_1e^{-c_2 N^\delta}.
\end{equation}
{F}rom the above equation we get $|\E_\nu(\lambda_k)-\E_\mu(\lambda_k)|$ is of order at most $N^{-1/2+\e}$,
so (\ref{eqn:concmu}) holds when replacing $\E_\nu(\lambda_k)$ by $\E_\mu(\lambda_k)$, proving concentration
at scale $N^{-1/2}$ for $\mu$.

Define $\gamma_k^{(N)}$  by
\be\label{gk}
\int_{-\infty}^{\gamma_k^{(N)}}\rho_1^{(N)}=\frac{k}{N}.
\ee
The proof will be complete if we can prove that for any $\e>0$ and
$k\in\llbracket \alpha N,(1-\alpha) N\rrbracket$, for large enough $N$ we have
\begin{equation}\label{eqn:initialaccuracy}
|\gamma_k^{(N)}-\gamma_k|<N^{-1/2+\e},
\end{equation}
By Lemma \ref{lem:Johansson}, $|m_N-m|\to 0$ for $\eta>N^{-1/2+\e}$, because on this domain
 $\frac{1}{N^2}k_N\to 0$, as concentration
at scale $N^{-1/2}$ holds for $\mu$. So using Lemma \ref{lem:HS} we finally get that
 (\ref{eqn:initialaccuracy}) holds, finishing the proof.
\end{proof}

\section{The multiscale analysis}

The purpose of this paragraph is to prove the following proposition:
 if rigidity holds at scale $N^{-1+a}$, it holds also at scale $N^{-1+\frac{3}{4}a}$.
 The argument very closely follows  Section 3.3 of \cite{BouErdYau2011} and
we will just explain the modifications.

\begin{proposition}\label{prop:induction}
 Assume that for some $a\in(0,1)$ the following property holds:
for any $\alpha,\e>0$, there are constants
$\delta,c_1,c_2>0$ such that for any $N\geq 1$ and $k\in\llbracket \alpha N,(1-\alpha) N\rrbracket$,
\begin{equation}\label{eqn:IndHyp}
\P_\mu\left(|\lambda_k-\gamma_k|> N^{-1+a+\e}\right)\leq c_1e^{-c_2N^\delta}.
\end{equation}
Then the same property holds also replacing $a$ by $3a/4$:
for any $\alpha,\e>0$, there are constants
$\delta,c_1,c_2>0$ such that for any $N\geq 1$ and $k\in\llbracket \alpha N,(1-\alpha) N\rrbracket$,
we have
$$
\P_\mu\left(|\lambda_k-\gamma_k|> N^{-1+\frac{3}{4}a+\e}\right)\leq c_1e^{-c_2N^\delta}.
$$
\end{proposition}

\noindent{\bf Proof of Theorem \ref{thm:rigidity}.}
This is an immediate consequence of the initial estimate, Proposition \ref{prop:initial}, and iterations
of Proposition \ref{prop:induction}.\hfill\qed
\vspace{0.2cm}

As in  Section 3.3 of \cite{BouErdYau2011},
two steps are required in the proof of the above Proposition \ref{prop:induction}.
First we  will prove that   concentration holds  at the smaller scale $N^{-1+\frac{a}{2}}$.

\begin{proposition}\label{prop:concentration}
Assume that (\ref{eqn:IndHyp}) holds. Then for any $\alpha>0$ and $\e>0$,
there are constants $c_1,c_2,\delta>0$ such that for any $N\geq 1$ and $k\in\llbracket \alpha N,(1-\alpha)N\rrbracket$,
\begin{equation}\label{eqn:conc}
\P_\mu\left(|\lambda_k-\E_\mu(\lambda_k)|>\frac{N^{\frac{a}{2}+\e}}{N}\right)\leq
c_1e^{-c_2N^\delta}.
\end{equation}
\end{proposition}

After the better concentration from this proposition, the rigidity can be improved to the scale $N^{-1+\frac{3a}{4}}$.

\begin{proposition}\label{prop:rigidity}
Assume that (\ref{eqn:IndHyp}) holds.
Then for any $\alpha>0$ and $\e>0$,
there is a constant $c>0$ such that for any $N\geq 1$ and
$k\in\llbracket \alpha N,(1-\alpha)N\rrbracket$,
$$
\left|\gamma_k^{(N)}-\gamma_k\right|\leq c\frac{N^{\frac{3a}{4}+\e}}{N}
$$
where $\gamma_k^{(N)}$ is defined in \eqref{gk}.
\end{proposition}

Propositions  \ref{prop:concentration} and \ref{prop:rigidity}  are the equivalent versions of
 Propositions 3.12 and 3.13 of  \cite{BouErdYau2011}
with no convexity assumption on $V$.
Proposition \ref{prop:induction} can be proved  exactly in the same way  as Proposition 3.11
   \cite{BouErdYau2011} by
using Propositions  \ref{prop:concentration} and \ref{prop:rigidity}.
Notice that this argument  does not use the convexity of $V$.  We now explain
the proof of Propositions  \ref{prop:concentration} and \ref{prop:rigidity}.

The convexity of $V$ is used critically in the proof of Proposition 3.12 of \cite{BouErdYau2011}.
Our measure $\mu$ is not convex, but thanks to Lemma~\ref{lem:equivalence}, it has the same exponentially small events
as the measure $\nu^{(s,\ell)}$  for any fixed $s, \ell>0$. Hence it suffices to  prove \eqref{eqn:conc}
with  $\mu$ replaced by $\nu^{(s,\ell)}$.
Choose an appropriate  $s, \ell$ such that
 the Hamiltonian of $\nu=\nu^{(s,\ell)}$ is convex (Lemma~\ref{lem:convex}).
Then the  proof of \eqref{eqn:conc} with $\mu$ replaced by $\nu$ is
identical to the proof of Proposition 3.12 of \cite{BouErdYau2011}  since the measure $\nu$ is now convex.

For the proof of Proposition~\ref{prop:rigidity}, we can
follow the proof of Proposition  3.13 in  \cite{BouErdYau2011}
line by line. At a single place, in estimating the second
term on the r.h.s. of (3.51) in  \cite{BouErdYau2011}, the spectral gap inequality for $\mu$ (Eq. (3.12) in
 \cite{BouErdYau2011}) was used, but the necessary estimate immediately follows
from Proposition~\ref{prop:initial}.

\end{document}